\documentclass[1p,final]{elsarticle}
\usepackage{amsfonts,natbib,pslatex}

\usepackage{latexsym,amsmath,amsthm,amssymb,amsxtra,mathrsfs,times,CJK}

\usepackage{color}
\usepackage{amsfonts,color}
\usepackage{amssymb,amsmath,latexsym}
\usepackage{amssymb}

\usepackage[para]{threeparttable}

\newtheorem{theorem}{Theorem}[section]

\newtheorem{corollary}[theorem]{Corollary}

\newtheorem{example}[theorem]{Example}

\newtheorem{remark}[theorem]{Remark}

\journal{Finite Fields and Their Applications}

\begin{document}

\begin{frontmatter}



\title{A New Criterion on Normal Bases of Finite Field Extensions}


\author[XUT]{Aixian Zhang\corref{cor1}}
 \ead{zhangaixian1008@126.com}
\author[TSH]{Keqin Feng}
 \ead{kfeng@math.tsinghua.edu.cn}

 \cortext[cor1]{Corresponding author}
 \address[XUT]{Department of Mathematical Sciences, Xi'an  University of Technology,
Shanxi, 710048, China.}
 \address[TSH]{Department of Mathematical
Sciences, Tsinghua University,
 Beijing, 100084, China.}

\begin{abstract}
A new criterion on normal bases of finite field extension $\mathbb{F}_{q^n} / \mathbb{F}_{q}$
is presented and explicit criterions for several particular finite field extensions are derived
from this new criterion.
\end{abstract}


\end{frontmatter}

\section{Introduction}
The determination of normal bases for finite field extensions is one of the important
topics in applications such as coding, cryptography and practical computation, particularly
multiplication operation in finite fields.

A series of criterions on normal bases has been given (\cite{LN,Wan}), many series of normal bases with lower
complexity have been found (\cite{ABV,CGPT,Gao,GL,Liao,LF,LY,Seguin,YP}), and explicit description to construct normal bases for specific cases of
finite field have been presented (\cite{BGM,GG,Kyuregyan,Peris,PWO,Samaev,WB}).

In this paper we present a new criterion on normal bases for general case of extension
$\mathbb{F}_{q^n} / \mathbb{F}_{q}$ in Section 2. As applications of this new criterion, we show several examples
in Section 3 which give explicit description of the normal bases for several specific extension of finite fields
including previous results in (\cite{Peris,PWO,WB}).

\section{A New Criterion on Normal Basis Generators for Finite Field Extensions}

Let $q=p^l$ be a power of prime number $p,l \geq 1, \ \mathbb{F}_{q}$ be the finite field with $q$ elements.
An element $ \alpha \in \mathbb{F}_{q^n}$ is called a normal basis generator (NBG) for extension
$\mathbb{F}_{q^n} / \mathbb{F}_{q}$ if $\mathbf{B}=\{\alpha, \alpha^q, \alpha^{q^2},\cdots,\alpha^{q^{n-1}}\}$
is a $\mathbb{F}_{q}$-basis of $\mathbb{F}_{q^n}$. In this case, $\mathbf{B}$ is called a normal basis for
$\mathbb{F}_{q^n} / \mathbb{F}_{q}$.

The normal bases for $\mathbb{F}_{q^n} / \mathbb{F}_{q}$ are closely related to the ring of $q$-polynomial
which we introduce now briefly. For more information  on normal bases and $q$-polynomial we refer to books \cite{LN,Wan}.

A $q$-polynomial (or called linearized polynomial) is a polynomial in the following form:
$$
L(x)=a_0x+a_1x^q+a_2x^{q^2}+\cdots+a_mx^{q^m} \ \ (a_i \in \mathbb{F}_{q}).
$$

Let $\mathcal{F}_q[x]$ be the set of all $q$-polynomials. Then $\mathcal{F}_q[x]$ is a ring with
respect to the ordinary addition and the following multiplication $\otimes$ :
$$
L(x) \otimes K(x)=L(K(x)) \ \ \ (\mbox{composition}).
$$
One of basic facts on $\mathcal{F}_q[x]$ is that the mapping
$$
\varphi:\mathbb{F}_{q}[x] \rightarrow \mathcal{F}_q[x], \ \sum^m_{i=0}a_ix^i \mapsto \sum^m_{i=0}a_ix^{q^i} \ (a_i \in \mathbb{F}_{q})
$$
is an isomorphism of rings. Therefore $\mathcal{F}_q[x]$ is a principal ideal domain with identity $x.$ We use the notation $\parallel$
to express the divisibility in $\mathcal{F}_q[x]$. Namely, for $L(x)$ and $M(x)$ in $\mathcal{F}_q[x]$, $L(x) \parallel M(x)$ means that
$L(x) \neq 0$ and there exists $N(x) \in \mathcal{F}_q[x]$ such that $M(x)=L(x) \otimes N(x)=N(x) \otimes L(x).$

Let $n$ be a positive integer. For $\alpha \in \mathbb{F}_{q^n},$ the set
$$
I_\alpha=\{M(x) \in \mathcal{F}_q[x]: M(\alpha)=0\}
$$
is a nonzero ideal of $\mathcal{F}_q[x]$ and $x^{q^n}-x \in I_\alpha.$ The monic generator $M_\alpha(x)$ of the ideal $I_\alpha$
is called the minimal $q$-polynomial of $\alpha$ over $\mathbb{F}_{q}.$ Particularly, $M_\alpha(x)$ is an irreducible polynomial
in $\mathcal{F}_q[x]$ and $M_\alpha(x) \parallel x^{q^n}-x .$ Let $x^n -1$ has the following standard decomposition in
$\mathbb{F}_{q}[x]:$
\begin{equation}\label{eqn-decom}
x^n -1=p_1(x)^{a_1}p_2(x)^{a_2}\cdots p_r(x)^{a_r},
\end{equation}
where $p_1(x),p_2(x),\cdots, p_r(x)$ are distinct monic irreducible polynomials in $\mathbb{F}_{q}[x]$
and $a_i \geq 1 \ (1 \leq i \leq r).$ Then the standard decomposition of $x^{q^n}-x =\varphi(x^n -1)$
in $\mathcal{F}_q[x]$ is
$$
x^{q^n}-x =P_1(x)^{a_1} \otimes P_2(x)^{a_2} \otimes \cdots  \otimes P_r(x)^{a_r},
$$
where  $P_i(x)=\varphi(p_i(x)) \ (1 \leq i \leq r)$ are distinct monic irreducible $q$-polynomials in
$\mathcal{F}_q[x]$.

An element $\alpha \in \mathbb{F}_{q^n}$ is a NBG of $\mathbb{F}_{q^n} / \mathbb{F}_{q}$ means,
by definition $ \{\alpha, \alpha^q, \alpha^{q^2},\cdots,\alpha^{q^{n-1}}\}$ is $\mathbb{F}_{q}$-linear independent.
This is also equivalent to that there is no non-zero $q$-polynomial $G(x)=\sum^{n-1}\limits_{i=0}c_ix^{q^i} \ (c_i \in \mathbb{F}_{q})$
such that $G(x) \parallel x^{q^n}-x$ and $G(\alpha)=0.$ From this we give the following usual criterions on $\alpha \in \mathbb{F}_{q^n}$
being a NBG of $\mathbb{F}_{q^n} / \mathbb{F}_{q}$.

\begin{theorem}\label{thm-oddNBG}(\cite{LN,Wan})
Suppose that $x^n-1$ has the decomposition (\ref{eqn-decom}) in $\mathbb{F}_q[x]$.
Let $l_i(x)=\frac{x^n -1}{p_i(x)}$ and $L_i(x)=\varphi(l_i(x)) \ (1 \leq i \leq r).$
Then for $\alpha \in \mathbb{F}_{q^n}, \ \alpha$ is a NBG of $\mathbb{F}_{q^n} / \mathbb{F}_{q}$
if and only if one of the following conditions satisfied

(1) The minimal $q$-polynomial $M_\alpha(x)$ of $\alpha$ is $x^{q^n}-x.$

(2) For each factor $m(x)$ of $x^n-1 $ in $\mathbb{F}_q[x]$ with degree $< n$ and
$M(x)=\varphi(m(x)), \ M(\alpha)\neq 0.$

(3) $L_i(\alpha)\neq 0 \ (1 \leq i \leq r).$
\end{theorem}

The criterions presented in Theorem \ref{thm-oddNBG} heavily depend on the decomposition
(\ref{eqn-decom}) of $x^n-1.$ Now we present a new criterion on NBG of $\mathbb{F}_{q^n} / \mathbb{F}_{q}$
which we use the $q$-equivalent classes of the elements in $\mathrm{Z}_n=\mathbb{Z}/ n\mathbb{Z}.$
To compute these $q$-equivalent classes is easier than to find the decomposition of $x^n-1$ in $\mathbb{F}_q[x]$.

Firstly we assume that $(n,q)=1$ (The other case can be easily reduced to $(n,q)=1$ case, see Theorem \ref{thm-redu}).
Then the decomposition of $x^n-1$ in $\mathbb{F}_q[x]$ is
\begin{equation}\label{eqn-coprimedecom}
x^n -1=p_1(x)p_2(x)\cdots p_r(x),
\end{equation}
where $p_i(x) \ (1 \leq i \leq r)$ are distinct monic irreducible polynomials in $\mathbb{F}_q[x]$.
The ring $\mathbf{R}=\mathbb{F}_q[x] / (x^n -1)$ is semi-simple and, by Chinese Remainder Theorem,
is a direct sum of finite fields:
\begin{equation}\label{eqn-CRT}
\mathbf{R}=\frac{\mathbb{F}_q[x]}{(x^n -1)} \cong \oplus^r_{i=1}\frac{\mathbb{F}_q[x]}{(p_i(x))}\cong \oplus^r_{i=1}\mathbb{F}_{q^{d_i}},
\end{equation}
where $d_i=\deg p_i(x) \ (1 \leq i \leq r).$ Let $\zeta$ be a fixed $n$-th primitive root of 1 in the algebraic
closure of $\mathbb{F}_q.$ Then $\mathrm{Z}_n$ is partitioned into $r$ $q$-classes
\begin{eqnarray}\label{eqn-parti}
\mathcal{S}_1 &=& \{a_1=0\} \nonumber \\
\mathcal{S}_2 &=& \{a_2, a_2q, \cdots, a_2q^{d_2-1}\}  \ \ (a_2q^{d_2}=a_2 \in \mathrm{Z}_n) \nonumber \\
&&\vdots  \nonumber \\
\mathcal{S}_r &=& \{a_r, a_r q, \cdots, a_r q^{d_r-1}\} \ \  (a_r q^{d_r}=a_r \in \mathrm{Z}_n)
\end{eqnarray}
and the roots $\{ 1, \zeta, \zeta^2, \cdots, \zeta^{n-1}\}$ of $x^n-1$ are partitioned into $r$ $\mathbb{F}_{q}$-conjugate
classes:
\begin{eqnarray*}
\mathcal{A}_1 &=& \{\alpha_1=1\} \\
\mathcal{A}_2 &=& \{\alpha_2, \alpha^q_2, \cdots, \alpha^{ q^{d_2-1}}_2\}  \\
&&\vdots \\
\mathcal{A}_r &=& \{\alpha_r, \alpha^q_r, \cdots, \alpha^{q^{d_r-1}}_r \},
\end{eqnarray*}
where $\alpha_i=\zeta^{a_i}, \ d_i=\deg p_i(x)$ and $\mathcal{A}_i$ is the set of roots of
$p_i(x) \ (1 \leq i \leq r, p_1(x)=x-1).$

Our new criterion on NBG for $\mathbb{F}_{q^n} / \mathbb{F}_{q}$ is expressed in terms of the
orthogonal idempotent elements $e_i(x) \in \mathbb{F}_{q}[x], \ \deg e_i(x) \leq n-1 \ (1 \leq i \leq r) $
satisfying
\begin{equation}\label{eqn-conj}
e_i(x) \equiv \delta_{ij} \ (\bmod  p_j(x)) \ \ \ (1 \leq i,j \leq r),
\end{equation}
where $\delta_{ij}$ is the Kronecker symbol.

By Chinese Remainder Theorem, such idempotents $e_i(x) \ (1 \leq i \leq r)$ exist and uniquely determined.
From (\ref{eqn-conj}) we get
\begin{equation}\label{eqn-equa}
e_i(\alpha_j)=\delta_{ij} \ (1 \leq i,j \leq r)
\end{equation}
and have the following orthogonal idempotent decomposition in $\mathbf{R}=\frac{\mathbb{F}_q[x]}{(x^n -1)},$
\begin{equation}\label{eqn-idem}
1=e_1(x)+\cdots+e_r(x), \ \ e_i(x)e_j(x)=\delta_{ij}e_i(x) \ (1 \leq i,j \leq r).
\end{equation}

Let $E_i(x)=\varphi(e_i(x)) \in \frac{\mathcal{F}_q[x]}{(x^{q^n} -x)} \ (1 \leq i \leq r).$
Our new criterion on NBG for $\mathbb{F}_{q^n} / \mathbb{F}_{q}$ is the following fundamental result.

\begin{theorem}\label{thm-newNBG}
Let $e_i(x) \ (1 \leq i \leq r)$ be the idempotent elements in $\mathbf{R}=\frac{\mathbb{F}_q[x]}{(x^n -1)}$
defined by (\ref{eqn-conj}). $E_i(x)=\varphi(e_i(x)) \in \frac{\mathcal{F}_q[x]}{(x^{q^n} -x)}.$ Then for
$\alpha \in \mathbb{F}_{q^n}, \ \alpha$ is a NBG for $\mathbb{F}_{q^n} / \mathbb{F}_{q}$ if and only if
$E_i(\alpha) \neq 0 \ (1 \leq i \leq r).$
\end{theorem}

\begin{proof}
Let $l_i(x)=\frac{x^n-1}{p_i(x)}, \ L_i(x)=\varphi(l_i(x)) \ (1 \leq i \leq r).$
From Theorem \ref{thm-oddNBG} we know that $\alpha$ is a NBG for $\mathbb{F}_{q^n} / \mathbb{F}_{q}$
if and only if $L_i(\alpha)\neq 0 \ (1 \leq i \leq r).$ Now we claim that for each $i \ (1 \leq i \leq r),$
$$
L_i(\alpha)=0 \Leftrightarrow E_i(\alpha)=0.
$$
From this the Theorem \ref{thm-newNBG} follows.

From (\ref{eqn-conj}) we get
$$ l_i(x) \equiv l_i(x)e_i(x)(\bmod x^n-1).$$
By the isomorphism $\varphi:\mathbb{F}_{q}[x] \rightarrow \mathcal{F}_q[x]$ we get the following congruences in
ring $\mathcal{F}_q[x]:$
$$
L_i(x) \equiv L_i(x)\otimes E_i(x)(\bmod x^{q^n}-x).
$$

Then by $\alpha^{q^n}=\alpha$ we get
$$
L_i(\alpha)=L_i(x)\otimes E_i(x)\mid_{x=\alpha}=L_i(E_i(\alpha)).
$$
Particularly, if $E_i(\alpha)=0,$ then $L_i(\alpha)=L_i(0)=0.$

Conversely, the fact $(l_i(x),p_i(x))=1$ in $\mathbb{F}_{q}[x]$ implies that there exist
$a(x), b(x) \in \mathbb{F}_{q}[x]$ such that $a(x)l_i(x)+b(x)p_i(x)=1.$ Therefore
\begin{eqnarray*}
e_i(x)&=&e_i(x)a(x)l_i(x)+e_i(x)b(x)p_i(x)\\
&\equiv & e_i(x)a(x)l_i(x) (\mbox{mod} x^n-1),
\end{eqnarray*}
since $e_i(x)p_i(x) \equiv 0(\bmod x^n-1)$ by (\ref{eqn-conj}). Therefore in $\mathcal{F}_q[x],
E_i(x) \equiv E_i(x) \otimes A(x)\otimes L_i(x) (\bmod x^{q^n}-x),$
where $A(x)=\varphi(a(x)).$ Therefore $L_i(\alpha)=0$ implies that $E_i(\alpha)=0.$
This completes the proof of Theorem \ref{thm-newNBG}.
\end{proof}

Next we present a rather easy method to compute the idempotents $e_i(x)$
and so $E_i(x)=\varphi(e_i(x)) \ (1 \leq i \leq r).$

\begin{theorem}\label{thm-idem}
Let
$$
\varepsilon_i(x)=\sum_{a \in \mathcal{S}_i}x^a
$$
where $\mathcal{S}_i \ (1 \leq i \leq r)$ are $q$-classes of $\mathrm{Z}_n$ are defined by
partition (\ref{eqn-parti}), $\alpha_i \in \mathcal{A}_i$ and $\mathbf{M}$ is an $r \times r$
matrix over $\mathbb{F}_{q}$ defined by
$$\mathbf{M}=(\varepsilon_i(\alpha_j))_{1 \leq i,j \leq r}.$$

Then $\det(\mathbf{M})\neq 0$ and

\begin{equation}\label{eqn-matrix}
\left(
          \begin{array}{c}
          e_1(x) \\
          \vdots \\
           e_r(x) \\
            \end{array}
            \right)=\mathbf{M^{-1}} \left(
          \begin{array}{c}
          \varepsilon_1(x) \\
          \vdots \\
           \varepsilon_r(x). \\
            \end{array}
            \right).
\end{equation}
\end{theorem}

\begin{proof}
Firstly we prove that $\varepsilon_i(\alpha_j) \in \mathbb{F}_{q} \ (1 \leq i,j \leq r).$
Since $\alpha_j=\zeta^{a_j}$, we have $\varepsilon_i(\alpha_j)=\sum^{d_i-1}\limits_{\lambda=0}\zeta^{a_ja_iq^\lambda}$ and
\begin{eqnarray*}
\varepsilon_i(\alpha_j)^q &=&\sum^{d_i-1}_{\lambda=0}(\zeta^{a_iq^{\lambda+1}})^{a_j}\\
&=&\sum^{d_i-1}_{\lambda=0}(\zeta^{a_iq^{\lambda}})^{a_j} = \varepsilon_i(\alpha_j).
\end{eqnarray*}
Therefore $\varepsilon_i(\alpha_j) \in \mathbb{F}_{q}$ and $\mathbf{M}$ is a matrix over $\mathbb{F}_{q}$.
By the definition of $\varepsilon_i(x)$ we know that
$$
\varepsilon_i(x) \equiv x^{a_i}+ x^{a_iq}+\cdots+x^{a_iq^{d_i-1}} (\bmod x^n-1) \ (1 \leq i \leq r).
$$
Then by (\ref{eqn-idem}) we have $e_i(x)^q \equiv e_i(x)(\bmod x^n-1).$ Therefore $e_i(x)$ is a
$\mathbb{F}_{q}$-linear combination of $\varepsilon_1(x),\varepsilon_2(x),\cdots,\varepsilon_r(x).$
Namely,
$$\left(
          \begin{array}{c}
          e_1(x) \\
          \vdots \\
           e_r(x) \\
            \end{array}
            \right)= \mathbf{A} \left(
          \begin{array}{c}
          \varepsilon_1(x) \\
          \vdots \\
           \varepsilon_r(x) \\
            \end{array}
            \right), $$
where $\mathbf{A}$ is an $r \times r$ matrix over $\mathbb{F}_{q}$. By using (\ref{eqn-equa}), we get
$$\mathbf{I}_r=\left(
                \begin{array}{ccc}
                  e_1(\alpha_1) & \cdots & e_1(\alpha_r) \\
                  \vdots &  & \vdots \\
                  e_r(\alpha_1) & \cdots & e_r(\alpha_r) \\
                \end{array}
              \right)=\mathbf{A}\left(
                \begin{array}{ccc}
                  \varepsilon_1(\alpha_1) & \cdots & \varepsilon_1(\alpha_r) \\
                  \vdots &  & \vdots \\
                  \varepsilon_r(\alpha_1) & \cdots & \varepsilon_r(\alpha_r) \\
                \end{array}
              \right)=\mathbf{AM}.
$$
Therefore $\det(\mathbf{M})\neq 0$ and $\mathbf{A}=\mathbf{M^{-1}}.$
This completes the proof of the Theorem \ref{thm-idem}.
\end{proof}

\section{Examples}

\begin{example}
Let $n$ be a prime number, $ q=p^m, \ n \neq p.$ Suppose that $q$ is a primitive root of $n$
which means that $(\mathbb{Z}/ n\mathbb{Z} )^\ast=\langle q\rangle$. Let $\zeta$ be an $n$-th
primitive root of 1 so that $\mathbb{F}_{q}(\zeta)=\mathbb{F}_{q^{n-1}}.$ Then $x^n-1$ is decomposed in
$\mathbb{F}_{q}[x]$ as
$$
x^n-1=(x-1)p_2(x),
$$
where $p_2(x)=x^{n-1}+x^{n-2}+\cdots+x+1$ is irreducible in $\mathbb{F}_{q}[x]$.
The $\mathbb{F}_{q}$-conjugate classes of $\{\zeta^\lambda: 0 \leq \lambda \leq n-1\}$ are
\begin{eqnarray*}
\mathcal{A}_1 &=&\{1\}, \ \alpha_1=1 \\
\mathcal{A}_2 &=& \{\zeta^{q^l}: 0 \leq l \leq n-2\}=\{\zeta^\lambda: 1 \leq \lambda \leq n-1\}.
\end{eqnarray*}
Therefore $\varepsilon_1(x)=1, \ \varepsilon_2(x)=x+x^2+\cdots+x^{n-1},$ and
$$
\mathbf{M}=\left(
             \begin{array}{cc}
               \varepsilon_1(1) &\varepsilon_1(\zeta) \\
               \varepsilon_2(1) &\varepsilon_2(\zeta) \\
             \end{array}
           \right)
           =\left(
             \begin{array}{cc}
               1 &1 \\
               n-1 & -1 \\
             \end{array}
           \right),
           \mathbf{M^{-1}}=\frac{1}{n}\left(
             \begin{array}{cc}
               1 &1 \\
               n-1 & -1 \\
             \end{array}
           \right).
$$
Therefore by (\ref{eqn-matrix}),
\begin{eqnarray*}
e_1(x)&=&\frac{1}{n}(\varepsilon_1(x)+\varepsilon_2(x))=\frac{1}{n}(1+x+x^2+\cdots+x^{n-1})\\
e_2(x)&=&\frac{1}{n}((n-1)\varepsilon_1(x)-\varepsilon_2(x))=1-\frac{1}{n}(1+x+x^2+\cdots+x^{n-1})=1-e_1(x),
\end{eqnarray*}
and
$$
E_1(x)=\frac{1}{n}(\sum^{n-1}_{i=0}x^{q^i}), \ E_2(x)=x-E_1(x).
$$
Let $\mathrm{Tr}$ be the trace mapping for $\mathbb{F}_{q^n} / \mathbb{F}_{q}.$
Namely, for $\alpha \in \mathbb{F}_{q^n},$
$$
\mathrm{Tr}(\alpha)=\sum^{n-1}_{i=0}\alpha^{q^i}=nE_1(\alpha).
$$
Therefore,
$$
E_1(\alpha)=\frac{1}{n}\mathrm{Tr}(\alpha), \ E_2(\alpha)=\alpha- \frac{1}{n}\mathrm{Tr}(\alpha).
$$

By Theorem \ref{thm-newNBG}, for $\alpha \in \mathbb{F}_{q^n},$
\begin{eqnarray*}
\alpha  \ \mbox{is \ a \ NBG \ for} \  \mathbb{F}_{q^n} / \mathbb{F}_{q}
&\Leftrightarrow & E_1(\alpha) \neq 0 \  \mbox{and} \  E_2(\alpha) \neq 0 \\
&\Leftrightarrow & \mathrm{Tr}(\alpha) \neq 0 \ \mbox{and} \  n\alpha \neq \mathrm{Tr}(\alpha) \in \mathbb{F}_{q} \\
&\Leftrightarrow & \mathrm{Tr}(\alpha) \neq 0 \ \mbox{and} \  \alpha \notin \mathbb{F}_{q} (\mbox{since}\  \alpha \in \mathbb{F}_{q} \
\mbox{implies \ that} \  n\alpha = \mathrm{Tr}(\alpha)).
\end{eqnarray*}
Therefore we get the following result given by Pei et al. in \cite{PWO}.
\end{example}

\begin{theorem}\label{thm-Pei}
Let $q=p^m, \ n$ be a prime number, $n \neq p.$ If $(\mathbb{Z}/ n\mathbb{Z} )^\ast=\langle q\rangle$.
Then for $\alpha \in \mathbb{F}_{q^n}, \ \alpha$ is a NBG for $\mathbb{F}_{q^n} / \mathbb{F}_{q}$
if and only if $ \alpha \notin \mathbb{F}_{q}$ and $\mathrm{Tr}(\alpha) \neq 0$, where $\mathrm{Tr}$
is the trace mapping for $\mathbb{F}_{q^n} / \mathbb{F}_{q}$.
\end{theorem}

\begin{example}
Let $n$ be an odd prime number, $n \neq p, \ q=p^m.$ Suppose that the (multiplicative) order of $q$
in $(\mathbb{Z}/ n\mathbb{Z} )^\ast$ is $l=\frac{\varphi(n)}{2}=\frac{n-1}{2}$ so that in
$(\mathbb{Z}/ n\mathbb{Z} )^\ast$,
$$
\{ q^\lambda: 0 \leq \lambda \leq l-1\}= \{1 \leq r \leq n-1:(\frac{r}{n})=1\},
$$
where $(\frac{r}{n})$ is the Legendre symbol.

Let $\zeta$ be an $n$-th primitive root of 1 in the algebraic closure of $\mathbb{F}_{q}$
so that $\mathbb{F}_{q}(\zeta)=\mathbb{F}_{q^l}.$ The conjugate classes of $\{1,\zeta, \cdots, \zeta^{n-1}\}$
are
\begin{eqnarray*}
\mathcal{A}_1 &=& \{1\}, \ \alpha_1=1, \\
\mathcal{A}_2 &=& \{ \zeta^r: 1 \leq r \leq n-1, (\frac{r}{n})=1 \}, \alpha_2=\zeta, \\
\mathcal{A}_3 &=& \{ \zeta^r: 1 \leq r \leq n-1, (\frac{r}{n})=-1 \}, \alpha_3=\zeta^g,
\end{eqnarray*}
where $g$ is a generator of the cyclic group $(\mathbb{Z}/ n\mathbb{Z} )^\ast$. Therefore

\begin{equation}\label{eqn-exap}
\varepsilon_1(x)=1, \varepsilon_2(x)=\sum^{n-1}_{\scriptstyle r=1 \atop \scriptstyle  (\frac{r}{n})=1  }x^r,
\varepsilon_3(x)=\sum^{n-1}_{\scriptstyle r=1 \atop \scriptstyle  (\frac{r}{n})=-1  }x^r.
\end{equation}

$$\mathbf{M}=\left(
              \begin{array}{ccc}
                \varepsilon_1(1) & \varepsilon_1(\zeta) & \varepsilon_1(\zeta^g) \\
                \varepsilon_2(1) & \varepsilon_2(\zeta) & \varepsilon_2(\zeta^g) \\
               \varepsilon_3(1) & \varepsilon_3(\zeta) & \varepsilon_3(\zeta^g) \\
              \end{array}
            \right)
=\left(
   \begin{array}{ccc}
     1 & 1 & 1 \\
     l & C & B \\
     l & B & C \\
   \end{array}
 \right),
$$
where
$$
C=\sum^{n-1}_{\scriptstyle r=1 \atop \scriptstyle  (\frac{r}{n})=1  }\zeta^r \in \mathbb{F}_{q},
B=\sum^{n-1}_{\scriptstyle r=1 \atop \scriptstyle  (\frac{r}{n})=-1  }\zeta^r =-1-C,
$$
and
$$
\mathbf{M^{-1}}=\frac{1}{n(B-C)}\left(
   \begin{array}{ccc}
     B-C & B-C & B-C \\
     l(B-C) & C-l & l-B \\
     l(B-C) & l-B & C-l \\
   \end{array}
 \right).
$$

From (\ref{eqn-matrix}) we get

\begin{eqnarray}\label{eqn-doub}
  \left \{
\begin{array}{lll}
n e_1(x)=\varepsilon_1(x) + \varepsilon_2(x) + \varepsilon_3(x)=\sum^{n-1}_{i=0}x^i \\
n(B-C)e_2(x)=l(B-C)+(C-l)\varepsilon_2(x) + (l-B)\varepsilon_3(x) \\
n(B-C)e_3(x)=l(B-C)+(l-B)\varepsilon_2(x) + (C-l)\varepsilon_3(x).
\end{array}
\right.
\end{eqnarray}

Case (I): $2 \nmid q$

Let $n^\ast=(\frac{-1}{n})n.$ Then $B-C=-\sum_{r=1}^{n-1}(\frac{r}{n})\zeta^r$
is the quadratic Gauss sum, but valued in $\mathbb{F}_{q^l}.$ We have
\begin{eqnarray*}
(B-C)^2 &=& \sum_{1 \leq r,s \leq n-1}(\frac{r s}{n})\zeta^{r+s} =\sum_{1 \leq t,s \leq n-1}(\frac{t}{n})\zeta^{s(t+1)}\\
&=& -\sum_{t \neq -1}(\frac{t}{n})+(n-1)(\frac{-1}{n})=n^\ast.
\end{eqnarray*}

Therefore $B-C=\mu\sqrt{n^\ast}, \ \mu \in \{1,-1\}.$ Then from $B+C=-1,$ we get
$$
B=\frac{1}{2}(-1+\mu \sqrt{n^\ast}), C=\frac{1}{2}(-1-\mu \sqrt{n^\ast}).
$$

By (\ref{eqn-doub}) and (\ref{eqn-exap}) we have
\begin{eqnarray*}
n e_1(x)&=&\sum^{n-1}_{i=0}x^i, \\
n \mu\sqrt{n^\ast}e_2(x)&=& l\mu\sqrt{n^\ast}+[(-\frac{1}{2}-\frac{\mu \sqrt{n^\ast}}{2})-l]\varepsilon_2(x)+
[l+(\frac{1}{2}-\frac{\mu \sqrt{n^\ast}}{2})]\varepsilon_3(x), \\
&=& l\mu\sqrt{n^\ast}+\frac{n}{2}(\varepsilon_3(x) -\varepsilon_2(x) )-\frac{\mu\sqrt{n^\ast}}{2}(\varepsilon_3(x) +\varepsilon_2(x) )\\
n \mu\sqrt{n^\ast}e_3(x)&=&  l\mu\sqrt{n^\ast}-\frac{n}{2}(\varepsilon_3(x) -\varepsilon_2(x) )-\frac{\mu\sqrt{n^\ast}}{2}(\varepsilon_3(x) +\varepsilon_2(x) ).
\end{eqnarray*}
Then $E_i(x)=\varphi(e_i(x)) \ (1 \leq i \leq 3)$ are
\begin{eqnarray*}
n E_1(x)&=& \sum^{n-1}_{i=0}x^{q^i}, \\
2n\sqrt{n^\ast}E_2(x)&=& 2l\sqrt{n^\ast}x-\mu n \sum^{n-1}_{r=1}(\frac{r}{n})x^{q^r}-\sqrt{n^\ast}\sum^{n-1}_{r=1}x^{q^r},\\
2n\sqrt{n^\ast}E_3(x)&=& 2l\sqrt{n^\ast}x+\mu n \sum^{n-1}_{r=1}(\frac{r}{n})x^{q^r}-\sqrt{n^\ast}\sum^{n-1}_{r=1}x^{q^r}.
\end{eqnarray*}

Let $\mathrm{Tr}$ be the trace mapping for $\mathbb{F}_{q^n} / \mathbb{F}_{q}$.
By Theorem \ref{thm-newNBG} we get, for $\alpha \in \mathbb{F}_{q^n},$
\begin{eqnarray*}
\alpha \ \mbox{is \ a \ NBG \ for} \  \mathbb{F}_{q^n} / \mathbb{F}_{q} &\Leftrightarrow & E_i(\alpha) \neq 0 \ (1 \leq i \leq 3)\\
&\Leftrightarrow & \mathrm{Tr}(\alpha) \neq 0 \  \mbox{and}
\ 2l\sqrt{n^\ast} \alpha-\sqrt{n^\ast}(\mathrm{Tr}(\alpha)-\alpha) \neq \pm n \sum^{n-1}_{r=1}(\frac{r}{n})\alpha^{q^r} \\
&\Leftrightarrow & \mathrm{Tr}(\alpha) \neq 0 \  \mbox{and}
\ n \sqrt{n^\ast} \alpha-\sqrt{n^\ast}\mathrm{Tr}(\alpha) \neq \pm n \sum^{n-1}_{r=1}(\frac{r}{n})\alpha^{q^r}.
\end{eqnarray*}

Case (II): $2 \mid q$.

In this case $B+C=B-C=1$ and by (\ref{eqn-doub})
\begin{eqnarray*}
n e_1(x)&=&\sum^{n-1}_{i=0}x^i, \\
ne_2(x)&=&l+(l+B)(\varepsilon_2(x) +\varepsilon_3(x))+\varepsilon_2(x),\\
ne_3(x)&=&l+(l+B+1)(\varepsilon_2(x) +\varepsilon_3(x))+\varepsilon_2(x).
\end{eqnarray*}

Therefore, for $\alpha \in \mathbb{F}_{q^n},$
\begin{eqnarray*}
n E_1(\alpha)&=& \mathrm{Tr}(\alpha), \\
n E_2(\alpha) &=& l \mathrm{Tr}(\alpha)+B(\mathrm{Tr}(\alpha)+\alpha)+A, \\
n E_3(\alpha) &=& l \mathrm{Tr}(\alpha)+(B+1)(\mathrm{Tr}(\alpha)+\alpha)+A,
\end{eqnarray*}
where $A=\sum^{n-1}\limits_{\scriptstyle r=1 \atop \scriptstyle  (\frac{r}{n})=1  }\alpha^{q^r}. $
Therefore for $ \alpha \in \mathbb{F}_{q^n},$
\begin{eqnarray*}
\alpha \  \mbox{is \ a\ NBG \ for} \ \mathbb{F}_{q^n} / \mathbb{F}_{q}
&  \Leftrightarrow & \mathrm{Tr}(\alpha) \neq 0,
A \neq l\mathrm{Tr}(\alpha)+B(\mathrm{Tr}(\alpha)+\alpha) \\
&&  \mbox{and} \ A \neq l\mathrm{Tr}(\alpha)+(B+1)(\mathrm{Tr}(\alpha)+\alpha).
\end{eqnarray*}

If $n \equiv \pm 1(\bmod 8),$ then $(\frac{2}{n})=1$ and $B^2=\sum^{n-1}\limits_{\scriptstyle r=1 \atop \scriptstyle  (\frac{r}{n})=1 }\zeta^{2r}=B.$
Therefore $ B \in \{ 0,1\}.$ If $n \equiv \pm 3(\bmod 8),$ then $(\frac{2}{n})=-1$ and $B^2=C=B+1.$ Therefore $B \in \{ \omega, \omega+1 \}$
where $\omega$ and $\omega+1$ are two roots of $x^2+x+1$ in $\mathbb{F}_{4}.$
Thus we get the following result.
\end{example}

\begin{theorem}\label{thm-newre1}
Let $n$ be an odd prime number, $n \neq p, \ q=p^m, \ \mathrm{Tr}$ be the trace mapping for
$\mathbb{F}_{q^n} / \mathbb{F}_{q}, \ n^\ast=(\frac{-1}{n})n.$
Suppose that the (multiplicative) order of $q$
in $(\mathbb{Z}/ n\mathbb{Z} )^\ast$ is $l=\frac{\varphi(n)}{2}=\frac{n-1}{2}$.
Then for $ \alpha \in \mathbb{F}_{q^n},$ we have the following criterion

(1) If $p \geq 3, \alpha$ is a NBG for $\mathbb{F}_{q^n} / \mathbb{F}_{q}$ if and only if
$\mathrm{Tr}(\alpha) \neq 0$ and $\sqrt{n^\ast}(n \alpha -\mathrm{Tr}(\alpha)) \neq \pm n\sum^{n-1}\limits_{r=1}(\frac{r}{n})\alpha^{q^r}.$

(2) If $p=2, \ \alpha$ is a NBG for $\mathbb{F}_{q^n} / \mathbb{F}_{q}$ if and only if
$\mathrm{Tr}(\alpha) \neq 0$ and
$$
  \left \{
\begin{array}{ll}
A \neq l \mathrm{Tr}(\alpha), (l+1)\mathrm{Tr}(\alpha)+\alpha, \  \mbox{for} \ n \equiv \pm 1(\bmod 8) \\
A \neq l \mathrm{Tr}(\alpha)+ \omega (\mathrm{Tr}(\alpha)+\alpha), l \mathrm{Tr}(\alpha)+ (\omega+1) (\mathrm{Tr}(\alpha)+\alpha),\  \mbox{for} \ n \equiv \pm 3(\bmod 8),
\end{array}
\right.
$$
where $A=\sum^{n-1}\limits_{\scriptstyle r=1 \atop \scriptstyle  (\frac{r}{n})=1  }\alpha^{q^r} $ and
$\{ \omega,\omega+1 \}$ are two roots of $x^2+x+1$ in $\mathbb{F}_{4}$.

Particularly, if $q=p=2,$ then $\alpha \in \mathbb{F}_{2^n}$ is a NBG for $\mathbb{F}_{2^n} / \mathbb{F}_{2}$ if and only if
$\mathrm{Tr}(\alpha) =1$ and $ \sum^{n-1}\limits_{\scriptstyle r=1 \atop \scriptstyle  (\frac{r}{n})=1} \alpha^{q^r} \neq l, l+1+\alpha,$
where $\mathrm{Tr}$ is the trace mapping for $\mathbb{F}_{2^n} / \mathbb{F}_{2}$.
\end{theorem}

\begin{example}(Generalization of Example 2)
Let $n$ and $p$ be distinct prime numbers, $q=p^m.$ Let $f$ be the order of $q$ in $(\mathbb{Z}/ n\mathbb{Z} )^\ast, \ n-1=ef.$
Then we have a generator of the cyclic group $(\mathbb{Z}/ n\mathbb{Z} )^\ast$ such that $q \equiv g^e(\bmod n)$
and $(\mathbb{Z}/ n\mathbb{Z} )^\ast$ is partitioned into cyclotomic classes
$$
\mathrm{C_i}=\{g^{i+ej}: 0 \leq j \leq f-1 \} \ \ ( 0 \leq i \leq e-1).
$$
Let $\zeta$ be an n-th primitive root of 1, $\mathbb{F}_{q}(\zeta)=\mathbb{F}_{q^f}.$ Then $\{\zeta ^a: 0 \leq a \leq n-1\}$
be partitioned into $e+1 \ \mathbb{F}_{q}$-conjugate classes
\begin{eqnarray*}
\mathcal{A}_\ast &=& \{1\}, \alpha_\ast=1 \\
\mathcal{A}_i &=& \{ \zeta^a: a \in \mathrm{C_i}\}, \ \alpha_i=\zeta^{g^i} \ (0 \leq i \leq e-1).
\end{eqnarray*}

Therefore
$$\varepsilon_\ast(x)=1, \ \varepsilon_i(x) \equiv \sum_{a \in \mathrm{C_i}}x^a(\bmod x^n-1).$$

Let $\varepsilon_i=\varepsilon_i(\zeta) \ (0 \leq i \leq e-1)$. We know that
$\varepsilon_i \in \mathbb{F}_{q}$ and
$$\varepsilon_i (\alpha_j)=\sum_{a \in \mathrm{C_i}}\zeta^{ag^{j}}=\varepsilon_{i+j}.$$
Therefore
$$
\mathbf{M}= \left(
    \begin{array}{ccccc}
      1 & 1 &         1 &        \cdots &            1 \\
      f & \varepsilon_0 & \varepsilon_1 & \cdots & \varepsilon_{e-1} \\
      f & \varepsilon_1 & \varepsilon_2 & \cdots & \varepsilon_0 \\
      \vdots & \vdots    & \vdots &  & \vdots \\
      f & \varepsilon_{e-1} & \varepsilon_0 & \cdots & \varepsilon_{e-2} \\
    \end{array}
  \right).
$$

By using the equality
\begin{eqnarray*}
\sum^{e-1}_{i=0}\varepsilon_i \varepsilon_{i+j}&=&\sum^{e-1}_{i=0}\sum_{a,b \in \mathrm{C_0}}\zeta^{g^i(a+g^j b)}\\
&=&\left \{
\begin{array}{ll}
n-f , \  \mbox{if} \ -1 \in \mathrm{C_j}  \ (\Leftrightarrow j \equiv \frac{e f}{2}(\bmod e))\\
-f, \ \ \mbox{otherwise}.
\end{array}
\right.
\end{eqnarray*}
We can get
$$
\mathbf{M^{-1}}=\frac{1}{n} \left(
    \begin{array}{ccccc}
      1 & 1 & 1 & \cdots &1 \\
      f & \varepsilon_c &\varepsilon_{c+1} & \cdots & \varepsilon_{c-1} \\
      f & \varepsilon_{c+1} & \varepsilon_{c+2} & \cdots & \varepsilon_c \\
      \vdots & \vdots & \vdots &  & \vdots \\
      f & \varepsilon_{c-1} & \varepsilon_c & \cdots & \varepsilon_{c-2} \\
    \end{array}
  \right),
$$
where $c \equiv \frac{ef}{2}(\bmod e),$ namely
$$
c =  \left \{
\begin{array}{ll}
\frac{e}{2}, & \mbox{if} \ 2 \mid e \ \mbox{and} \ 2 \nmid f, \\
0 , & \mbox{otherwise} .
\end{array}
\right.
$$

Therefore for $ \alpha \in \mathbb{F}_{q^n}$,
\begin{eqnarray*}
n E_\ast(\alpha)&=&\mathrm{Tr}(\alpha),\\
n E_j(\alpha)&=& f \alpha+\sum^{e-1}_{i=0}\varepsilon_{i+j}\sum_{a \in \mathrm{C_i}}\alpha^{q^a}  \ (0 \leq j \leq e-1).
\end{eqnarray*}
\end{example}

Thus we get the following result.
\begin{theorem}\label{thm-new3}
Let $n$ and $p$ be distinct prime numbers, $q=p^m.$ Let $f$ be the order of $q$ in $(\mathbb{Z}/ n\mathbb{Z} )^\ast, \ n-1=ef.$
Let $\zeta$ be an n-th primitive root of 1, $\mathbb{F}_{q}(\zeta)=\mathbb{F}_{q^f}.$ We choose $g \in \mathbb{Z}$
such that $(\mathbb{Z}/ n\mathbb{Z} )^\ast = \langle g \rangle$ and $q \equiv g^e(\bmod n)$.
We denote
$$
\mathrm{C_i}=\{g^{i+ej}: 0 \leq j \leq f-1 \} \ \ ( 0 \leq i \leq e-1),
\varepsilon_i=\sum_{a \in \mathrm{C_i}}\zeta^a \in \mathbb{F}_{q} (0 \leq i \leq e-1).
$$

Then for $ \alpha \in \mathbb{F}_{q^n}, \ \alpha$ is a NBG for $\mathbb{F}_{q^n} / \mathbb{F}_{q}$
if and only if $\mathrm{Tr}(\alpha) \neq 0$ and
$$
\sum^{e-1}_{i=0}\varepsilon_{i+j}\sum_{a \in \mathrm{C_i}}\alpha^{q^a} \neq -f\alpha \ (0 \leq j \leq e-1).
$$
\end{theorem}

\begin{remark}
(1) The Gauss periods $\varepsilon_i \ (0 \leq i \leq e-1)$ can be calculated explicitly for many cases of
$e \ (=2,3,4,\cdots)$ by using Gauss sums so that more closer form of Theorem \ref{thm-new3} can be derived for such $e$.
Particularly, for $e=1$ and 2, we obtain Theorems \ref{thm-Pei} and \ref{thm-newre1}.

(2 ) For $q=2, e=3,5,7$ and $q=4, e=3,$ we have $\varepsilon_i \in \mathbb{F}_{2}$ and
$$
\sum^{e-1}_{i=0}\varepsilon_{i}\varepsilon_{i+j} =  \left \{
\begin{array}{ll}
1, & \mbox{if} \ j=0, \\
0 & \mbox{if} \ 1 \leq j \leq e-1,
\end{array}
\right.
$$

which means that the matrix

$$
\left(
  \begin{array}{cccc}
   \varepsilon_0 &\varepsilon_1 & \cdots & \varepsilon_{e-1} \\
    \varepsilon_1 &\varepsilon_2 &\cdots & \varepsilon_0 \\
    \vdots & \vdots &  &\vdots \\
   \varepsilon_{e-1} & \varepsilon_0 &\cdots &\varepsilon_{e-2} \\
  \end{array}
\right)
$$
is an orthogonal circulate matrix over $\mathbb{F}_{2}$.
Jungnickel et al. \cite{JBG} obtained a formula on the number of orthogonal circulate $e \times e$ matrices over $\mathbb{F}_{q}$.
From this formula we know that there essentially exist unique such matrix for $q=2,e=3,5,7$ and $q=4,e=3.$ Namely,
$(\varepsilon_0,\varepsilon_1,\cdots,\varepsilon_{e-1})=(1,0,\cdots,0).$

In these cases, the conclusion of Theorem \ref{thm-new3} can be simplified as :

$\alpha$ is a NBG for $\mathbb{F}_{q^n} / \mathbb{F}_{q}$ if and only if
$\mathrm{Tr}(\alpha) \neq 0$ and
$$
\sum^{e-1}_{i=0}\varepsilon_{i+j}\sum_{a \in \mathrm{C_i}}\alpha^{q^a} \neq -f\alpha \ (0 \leq j \leq e-1).
$$
\end{remark}

\begin{example}
Let $p_1, p_2,p$ be distinct prime numbers, $p_1 \geq 3, p_2 \geq 3, n=p_1p_2, q=p^m.$
Suppose that $(\mathbb{Z}/ p_1 \mathbb{Z} )^\ast =\langle q \rangle, (\mathbb{Z}/ p_2 \mathbb{Z} )^\ast =\langle q \rangle$
and $(p_1-1,p_2 -1)=2.$ Then the order of $q$ in $ (\mathbb{Z}/ n \mathbb{Z} )^\ast $ is $f=\frac{(p_1-1)(p_2 -1)}{2}.$
Let $\zeta$ be a n-th root of 1 and $\mathbb{F}_{q}(\zeta)=\mathbb{F}_{q^f}.$ The set $\{ \zeta^i :0 \leq i \leq n-1\}$
be partitioned into five $\mathbb{F}_{q}$-conjugate classes as following:
\begin{eqnarray*}
\mathcal{A}_0 &=& \{1\}, \  |\mathcal{A}_0|=1, \ \alpha_0=1, \\
\mathcal{A}_1 &=&  \{ \zeta, \zeta^q, \cdots, \zeta^{q^{f-1}}\}, \ |\mathcal{A}_1|=f, \ \alpha_1=\zeta, \\
\mathcal{A}_2 &=&  \{ \zeta^g, \zeta^{qg}, \cdots, \zeta^{q^{f-1}g}\}, \ |\mathcal{A}_2|=f, \ \alpha_2=\zeta^g (\mbox{where}\  \zeta^g \notin \mathcal{A}_1 \ \mbox{and} \ (g,n)=1), \\
\mathcal{A}_3 &=&  \{ \zeta^{p_1}, \zeta^{qp_1}, \cdots, \zeta^{q^{p_2 -2}p_1}\}, \ |\mathcal{A}_3|=p_2 -1, \ \alpha_3=\zeta^{p_1},\\
\mathcal{A}_4 &=&  \{ \zeta^{p_2}, \zeta^{qp_2}, \cdots, \zeta^{q^{p_1 -2}p_2}\}, \ |\mathcal{A}_4|=p_1 -1, \ \alpha_4=\zeta^{p_2}.
\end{eqnarray*}

Therefore
\begin{eqnarray*}
\varepsilon_0(x)&=& 1, \varepsilon_1(x)=\sum^{f-1}_{i=0}x^{q^i}, \varepsilon_2(x)=\sum^{f-1}_{i=0}x^{g q^i}, \\
\varepsilon_3(x)&=& \sum^{p_2-2}_{j=0}x^{q^j p_1}, \varepsilon_4(x)=\sum^{p_1-2}_{j=0}x^{q^j p_2}.
\end{eqnarray*}
$$
\mathcal{M}=(\varepsilon_i(\alpha_j))_{0 \leq i,j \leq 4}=\left(
                                                            \begin{array}{ccccc}
                                                              1 & 1 & 1 & 1 & 1 \\
                                                              f & \varepsilon_1 & \varepsilon_2 & -\frac{p_1 -1}{2} & -\frac{p_2 -1}{2} \\
                                                             f & \varepsilon_2 & \varepsilon_1 & -\frac{p_1 -1}{2} & -\frac{p_2 -1}{2} \\
                                                              p_2 -1 & -1 & -1 & -1 & p_2 -1 \\
                                                              p_1 -1 & -1 & -1 & p_1 -1 & -1 \\
                                                            \end{array}
                                                          \right),
$$
where $\varepsilon_i=\varepsilon_i(\alpha_j) \in \mathbb{F}_{q} \ (i=1,2).$

If $ 2 \nmid q,$ then
$$n \mathcal{M}^{-1}=\left(
                       \begin{array}{ccccc}
                         1 & 1 & 1 & 1 & 1 \\
                         2m_1m_2 & \frac{n}{2(\varepsilon_1 -\varepsilon_2)}+\frac{1}{2} & -\frac{n}{2(\varepsilon_1 -\varepsilon_2)}+\frac{1}{2} & -m_1 & -m_2 \\
                         2m_1m_2 & -\frac{n}{2(\varepsilon_1 -\varepsilon_2)}+\frac{1}{2} & \frac{n}{2(\varepsilon_1 -\varepsilon_2)}+\frac{1}{2} & -m_1 & -m_2 \\
                         2m_2 & -1 & -1 & -1 & 2m_2 \\
                         2m_1 & -1 & -1 & 2m_1 & -1 \\
                       \end{array}
                     \right),
$$
where $m_i=\frac{p_i -1}{2} \ (i=1,2).$ Therefore
\begin{eqnarray*}
n e_0(x)&=&\sum^{n-1}_{i=0}x^{i}, \\
n e_1(x)&=& 2m_1m_2+\frac{n}{2(\varepsilon_1 -\varepsilon_2)}(\varepsilon_1(x)-\varepsilon_2(x))+\frac{1}{2}(\varepsilon_1(x)+\varepsilon_2(x))-m_1\varepsilon_3(x)-m_2 \varepsilon_4(x),\\
n e_2(x)&=& 2m_1m_2-\frac{n}{2(\varepsilon_1 -\varepsilon_2)}(\varepsilon_1(x)-\varepsilon_2(x))+\frac{1}{2}(\varepsilon_1(x)+\varepsilon_2(x))-m_1\varepsilon_3(x)-m_2 \varepsilon_4(x),\\
n e_3(x)&=&- \sum^{n-1}_{i=0}x^{i}+p_2(1+\varepsilon_4(x)),\\
n e_4(x)&=&- \sum^{n-1}_{i=0}x^{i}+p_1(1+\varepsilon_3(x)).
\end{eqnarray*}

For $\alpha \in \mathbb{F}_{q^n}$ and $m | n,$ let $\mathrm{Tr}_m^n(\alpha)$ be the trace of $\alpha$
for extension $\mathbb{F}_{q^n} / \mathbb{F}_{q^m}.$ We have, for $E_i(x)=\varphi(e_i(x)),$
\begin{eqnarray*}
n E_0(\alpha) &=& \mathrm{Tr}^n_1(\alpha),\\
n E_1(\alpha) &=& 2m_1m_2 \alpha+\frac{1}{2}\sum^{n-1}_{\scriptstyle r=1 \atop \scriptstyle  (r,n)=1}\alpha^{q^r}-m_1(\mathrm{Tr}_{p_1}^n(\alpha)-\alpha)
                - m_2(\mathrm{Tr}_{p_2}^n(\alpha)-\alpha)\\
                &&+\frac{n}{2(\varepsilon_1 - \varepsilon_2)}(\sum^{f-1}_{r=0}\alpha^{q^{q^r}}-\sum^{f-1}_{r=0}\alpha^{q^{g q^r}}),\\
n E_2(\alpha) &=& 2m_1m_2 \alpha+\frac{1}{2}\sum^{n-1}_{\scriptstyle r=1 \atop \scriptstyle  (r,n)=1}\alpha^{q^r}-m_1(\mathrm{Tr}_{p_1}^n(\alpha)-\alpha)
                - m_2(\mathrm{Tr}_{p_2}^n(\alpha)-\alpha)\\
                &&-\frac{n}{2(\varepsilon_1 - \varepsilon_2)}(\sum^{f-1}_{r=0}\alpha^{q^{q^r}}-\sum^{f-1}_{r=0}\alpha^{q^{g q^r}}),\\
n E_3(\alpha) &=& - \mathrm{Tr}^n_1(\alpha)+p_2 \mathrm{Tr}_{p_1}^n(\alpha), \\
n E_4(\alpha) &=& - \mathrm{Tr}^n_1(\alpha)+p_1 \mathrm{Tr}_{p_2}^n(\alpha).
\end{eqnarray*}

If $2 \mid q,$ then
$$
\mathcal{M}^{-1}=\left(
                   \begin{array}{ccccc}
                     1 & 1 & 1 & 1 & 1 \\
                     0 & \frac{n+1}{2}+\varepsilon_2 & \frac{n+1}{2}+\varepsilon_1 & m_1 & m_2 \\
                     0 & \frac{n+1}{2}+\varepsilon_1 & \frac{n+1}{2}+\varepsilon_2 & m_1 & m_2 \\
                     0 & 1 & 1 & 1 & 0 \\
                     0 & 1 & 1 & 0 & 1 \\
                   \end{array}
                 \right)
$$
and
\begin{eqnarray*}
E_0(\alpha)&=&\mathrm{Tr}_1^n(\alpha) \\
E_1(\alpha)&=& m_1(\mathrm{Tr}_{p_1}^n(\alpha)+\alpha)+m_2 (\mathrm{Tr}_{p_2}^n(\alpha)+\alpha)+\sum^{f-1}_{r=0}\alpha^{q^{q^r}}\\
E_2(\alpha)&=& m_1(\mathrm{Tr}_{p_1}^n(\alpha)+\alpha)+m_2 (\mathrm{Tr}_{p_2}^n(\alpha)+\alpha)+\sum^{f-1}_{r=0}\alpha^{q^{q^r}g}\\
E_3(\alpha)&=& \mathrm{Tr}_{p_1}^n(\alpha), \ \ \ E_4(\alpha)= \mathrm{Tr}_{p_2}^n(\alpha).
\end{eqnarray*}
\end{example}

Then Theorem \ref{thm-newNBG} implies the following result.

\begin{theorem}\label{thm-last}
Let $p_1, p_2,p$ be distinct prime numbers, $p_1 \geq 3, p_2 \geq 3, n=p_1p_2, q=p^m, f=\frac{(p_1-1)(p_2 -1)}{2},
m_1=\frac{p_1 -1}{2}, m_2=\frac{p_2 -1}{2}.$ Suppose that $(\mathbb{Z}/ p_i \mathbb{Z} )^\ast =\langle q \rangle \ (i=1,2)$
and $(p_1-1,p_2 -1)=2.$ We choose $g \in \mathbb{Z}$ such that $ g \notin  \langle q \rangle \subseteq (\mathbb{Z}/ n \mathbb{Z} )^\ast.$
Let $\zeta$ be an n-th root of 1 in the algebraic closure of $\mathbb{F}_{q}, \varepsilon_1=\sum^{f-1}_{r=0}\zeta^{q^i},\varepsilon_2=\sum^{f-1}_{r=0}\zeta^{g q^i}=1-\varepsilon_1.$ For $\alpha \in \mathbb{F}_{q^n}$ and $m \mid n,$ let $\mathrm{Tr}_m^n(\alpha)$ be the trace mapping of $\alpha$ for
$\mathbb{F}_{q^n}/ \mathbb{F}_{q^m}.$

(1) If $2\nmid q,$ then $\alpha \in \mathbb{F}_{q^n}$ is a NBG for $\mathbb{F}_{q^n}/ \mathbb{F}_{q}$ if and only if
$$
\mathrm{Tr}^n_1(\alpha) \neq 0, \  p_1\mathrm{Tr}^{n}_{p_2}(\alpha), \ p_2\mathrm{Tr}^{n}_{p_1}(\alpha)
$$
and
\begin{eqnarray*}
&&      2m_1m_2 \alpha+\frac{1}{2}\sum^{n-1}_{\scriptstyle r=1 \atop \scriptstyle  (r,n)=1}\alpha^{q^r}
                -m_1(\mathrm{Tr}_{p_1}^n(\alpha)-\alpha)
                - m_2(\mathrm{Tr}_{p_2}^n(\alpha)-\alpha)\\
& \neq & \pm \frac{n}{2(\varepsilon_1 - \varepsilon_2)}(\sum^{f-1}_{r=0}\alpha^{q^r}-\sum^{f-1}_{r=0}\alpha^{g q^r}).
\end{eqnarray*}

(2) If $2  \mid  q$, then $\alpha \in \mathbb{F}_{q^n}$ is a NBG for $\mathbb{F}_{q^n}/ \mathbb{F}_{q}$ if and only if
$$
\mathrm{Tr}^n_1(\alpha) \neq 0, \ \ \ \alpha \notin \mathbb{F}_{q^{p_i}} \ \ (i=1,2)
$$
and
$$
m_1(\mathrm{Tr}_{p_1}^n(\alpha)+\alpha)
                + m_2(\mathrm{Tr}_{p_2}^n(\alpha)+\alpha) \neq \sum^{f-1}_{r=0}\alpha^{q^{q^r c}} \ (c=1,g).
$$
\end{theorem}

At the end of this section we show that the case $p \mid n$ can be reduced into the case $p \nmid n.$

\begin{theorem}\label{thm-redu}
Let $n=p^t l, (l,p)=1, t \geq 1, q=p^m $ and $\mathrm{Tr}^{n}_{l}$ be the trace mapping for $\mathbb{F}_{q^n}/ \mathbb{F}_{q^l}.$
Then for $\alpha \in \mathbb{F}_{q^n}, \alpha$ is a NBG for $\mathbb{F}_{q^n}/ \mathbb{F}_{q}$ if and only if
$\mathrm{Tr}^{n}_{l}(\alpha)$ is a NBG for $\mathbb{F}_{q^l}/ \mathbb{F}_{q}.$
\end{theorem}

\begin{proof}
Let $$x^l -1=f_1(x)f_2(x)\cdots f_r(x),$$
where $f_i(x) \ (1 \leq i \leq r)$ are distinct monic irreducible polynomials in $\mathbb{F}_{q}[x].$
Then $x^n -1=(f_1(x)f_2(x)\cdots f_r(x))^{p^t}$ and Theorem \ref{thm-oddNBG} implies that
\begin{eqnarray*}
\alpha  \ \mbox{is\ a \ NBG \ for} \  \mathbb{F}_{q^n}/ \mathbb{F}_{q} &\Leftrightarrow & \alpha \ \mbox{is \ not \ a \ root \ of  } \varphi(\frac{x^n -1}{f_i(x)}) \ (1 \leq i \leq r) \\
& \Leftrightarrow & \alpha  \ \mbox{is \ not \ a \ root \ of } \varphi(l_i(x)(1+x^l+x^{2l}+\cdots+x^{(p^t -1)l})) \\
&& \mbox{where} \  l_i(x)=(x^l -1)/f_i(x) \ \ ( 1 \leq i \leq r) \\
& \Leftrightarrow & \alpha  \ \mbox{is \ not \ a \ root \ of \ } L_i(x) \otimes (x+x^{q^l}+x^{q^{2l}}+\cdots+x^{q^{(p^t -1)l}})\\
& \Leftrightarrow & L_i(\mathrm{Tr}^{n}_{l}(\alpha)) \neq 0 \ \ ( 1 \leq i \leq r) \\
& \Leftrightarrow & \mathrm{Tr}^{n}_{l}(\alpha) \ \mbox{is \ a \ NBG \ for \ } \ \mathbb{F}_{q^l}/ \mathbb{F}_{q}.
\end{eqnarray*}
\end{proof}
\begin{remark}
By combination of Theorem \ref{thm-redu} and Theorem \ref{thm-Pei}, we get the following result given by Peris \cite{Peris}.
\end{remark}

\begin{corollary}(\cite{Peris}) Let $q=p^l$ and $n=p^s$ be powers of prime number $p$ and $l,s \geq 1.$
Then $\alpha \in \mathbb{F}_{q^n}$ is a NBG for $\mathbb{F}_{q^n}/ \mathbb{F}_{q}$ if and only if
$\mathrm{Tr}(\alpha) \neq 0$, where $\mathrm{Tr}$ is the trace mapping for $\mathbb{F}_{q^n}/ \mathbb{F}_{q}$.
\end{corollary}

\section*{Acknowledgements}

K.Feng's research was supported by the Tsinghua National Lab. for Information Science and
Technology, and by the Science and Technology on Information Assurance Laboratory (No.KJ-12-01).

\end{document}